\documentclass[11pt]{amsart}

\usepackage{amsmath, amsfonts, amssymb,amsthm}
\usepackage{amstext}
\usepackage{mathrsfs}

\usepackage{float,epsf,subfigure}
\usepackage[all,cmtip]{xy}
\usepackage{hyperref}
\usepackage{algorithm}
\usepackage{algorithmic}
\usepackage{mathtools}
\usepackage{float}
\usepackage{bm}
\theoremstyle{plain}
\newtheorem{theorem}{Theorem}[section]

\newtheorem{cor}[theorem]{Corollary}

\newtheorem{def-thm}[theorem]{Definition-Theorem}
\newtheorem{lemma}[theorem]{Lemma}
\newtheorem{property}[theorem]{Property}

\newtheorem*{tha}{Theorem A}
\newtheorem*{thb}{Theorem B}

\theoremstyle{definition}

\begin{document}
\title[An Extension  of Zalcman's Lemma  for Complex  Lie Groups]{An Extension of Zalcman's Lemma  for Complex  Lie Groups}
\author[X.-J. Dong and Y.-D. Lv]{Xianjing Dong*, Yanda Lv and Lanjuan Xu}

\address{School of Mathematical Sciences \\ Qufu Normal University \\ Qufu, Jining, Shandong, 273165, P. R. China}
\email{xjdong05@126.com}
\address{School of Mathematical Sciences \\ Qufu Normal University \\ Qufu, Jining, Shandong, 273165, P. R. China}
\email{lvyanda2000@163.com}
\address{School of Mathematical Sciences \\ Beijing Normal University \\  Beijing, 100875, P. R. China}
\email{ljxu1494@bnu.edu.cn}


\thanks{$^*$Correspondence author}
\maketitle \thispagestyle{empty} \setcounter{page}{1}

\begin{abstract}   Zalcman's lemma makes significant applications in normal families, complex dynamics  and many other   related   problems in complex analysis. 
It is interesting to explore  a   global version of Zalcman's  lemma for complex manifolds, even  though  normality is a local property. 
 In this   paper, we establish   a version of 
    Zalcman's lemma for  complex  Lie groups  via  
  exponential mappings. 
 \end{abstract}

\vskip\baselineskip

\setlength\arraycolsep{2pt}
\medskip

\section{Introduction}

In 1975, Zalcman \cite{LZ1975} contributed an important   result   in characterization of normality of families of meromorphic functions on $\mathbb C.$ He  proved that  
\begin{tha}[Zalcman's Lemma] Let $\mathscr F$ be a family of meromorphic functions on a domain $\Omega\subset\mathbb C.$ Then, $\mathscr F$ is not  normal on $\Omega$  if and only if there exist 

$(a)$ a  sequence $\{f_j\}_{j=1}^{+\infty}\subset \mathscr F;$ 

$(b)$   a sequence $\{p_j\}_{j=1}^{+\infty}\subset\Omega$ with $p_j\to p_0\in\Omega$ as $j\to+\infty;$

$(c)$ a  sequence $\{\rho_j\}_{j=1}^{+\infty}\subset\mathbb R$ with  $0<\rho_j\to0$ as $j\to+\infty$ 

\noindent	 such that  the sequence 
$$\phi_j(z):=f_j(p_j+\rho_jz), \ \ \ \    j=1,2,\cdots$$
converges uniformly on compact subsets of $\mathbb C$ to a nonconstant meromorphic  function $\phi$ on $\mathbb C.$ 
\end{tha}

Zalcman's work was inspired by the result of Lohwater and Pommerenke \cite{L-P}, which establishes a well-known heuristic principle  showing that Bloch's  principle \cite{Blo} is valid to a certain  Picard type property.
Due to the importance of Zalcman's lemma, many authors have been working on   its  generalizations. 
A beautiful  work on this lemma  is due to Pang \cite{Pang, Pang1} who proved  a  stronger  version with a  real parameter, 
 which  thus shows that Bloch's  principle holds for some Picard type properties  related to derivatives. 
More details for    this  lemma and its applications  in one complex variable  
  refer to  \cite{CG, G-P-F, Sch, W-C, LZ1} and  references therein.

Many  efforts  were  made  to 
  generalize  Zalcman's  lemma to several complex variables, and more generally to complex manifolds.  An earlier  study about this lemma  was carried by Aladro and Krantz \cite{A-K}. Unfortunately, their result (Theorem 3.1, \cite{A-K}) was later pointed out  to be incorrect by Thai, Trang and Huong \cite{T-T-H} who presented  a correct extension of   Zalcman's  lemma as follows.

\begin{thb} Let $N$ be a compact Hermitian manifold, and $\Omega$ a domain in $\mathbb C^m.$
Let $\mathscr F: \Omega\to N$ be a family of holomorphic mappings.  Then, $\mathscr F$ is not  normal on $\Omega$  if and only if there exist 

$(a)$ a  sequence $\{f_j\}_{j=1}^{+\infty}\subset \mathscr F;$ 

$(b)$   a sequence $\{p_j\}_{j=1}^{+\infty}\subset\Omega$ with $p_j\to p_0\in\Omega$ as $j\to+\infty;$

$(c)$ a  sequence $\{\rho_j\}_{j=1}^{+\infty}\subset\mathbb R$ with  $0<\rho_j\to0$ as $j\to+\infty$ 

\noindent	 such that  the sequence 
$$\phi_j(z):=f_j(p_j+\rho_jz), \ \ \ \    j=1,2,\cdots$$
converges uniformly on compact subsets of $\mathbb C^m$ to a nonconstant holomorphic mapping  $\phi: \mathbb C^m\to N.$
\end{thb}

For $N=\mathbb P^1(\mathbb C),$  Theorem B  was  generalized  by  Chang, Xu and Yang \cite{C-X-Y}, Dovbush \cite{PD},   Yang and Pang \cite{Y-P},   etc.  For more details, we  refer to  \cite{D-K} and  the references therein.  

Note that  normality is a local property  and a  complex manifold is locally a domain in a complex Euclidean space,    one  can use  Zalcam's lemma locally to analyze   the normality of families of holomorphic mappings around  a point in complex manifolds. However, 
 it may be both   convenient and interesting to establish   a global version of   Zalcam's lemma for  complex manifolds in space   of Theorem B.
 Unfortunately,  unlike   Euclidean spaces,  a   general complex   manifold is   not a  linear space, 
which  lacks    addition and scalar multiplication.  So, we need to  seek  an appropriate  method  for extending Zalcman's lemma. 

In this  paper, we develop   a           Zalcman's  lemma for     complex Lie groups using   exponential mappings. 
           Note   the  examples for  complex Lie groups including 
        $\mathbb C^m, \mathbb C^m/\Lambda, GL(m, \mathbb C), SL(m,\mathbb C), O(m,\mathbb C)$  (see \cite{Lee}). 
Let  $G$ be a complex Lie group of finite dimension, whose  identity element  is denoted by $e.$  Note that every  $g\in G$  defines   a left translation  
 $L_g: G\to G$   via  
$L_g(h)=gh$ for $h\in G.$
Let  $T^{1,0}_gG$ be the holomorphic tangent space of $G$ at $g\in G.$ The \emph{holomorphic Lie algebra} of $G$ is defined by 
$\mathfrak g=T^{1,0}_eG.$
 Let $$\exp_g: \  T^{1,0}_gG\to G, \ \ \ \     g\in G$$ be the exponential mapping which is a holomorphic mapping defined by  the holomorphic one-parameter subgroup of $G,$ see  Section \ref{sec2} for details.  

We establish  a Zalcman's lemma for complex Lie groups. 
\begin{theorem}\label{main} Let $G$ be an  $m$-dimensional complex Lie group with holomorphic Lie algebra 
$\mathfrak g,$ 
and $\Omega$ a domain in $G.$ Fix  a  basis  
   $\varepsilon=(\varepsilon_1,\cdots, \varepsilon_m)$ of $\mathfrak g.$
Let $N$ be a compact Hermitian manifold. Let $\mathscr F: \Omega\to N$ be a family of holomorphic mappings. 
Then,  $\mathscr F$ is not normal on $\Omega$  if and only if  there exist 

$(a)$ a  sequence $\{f_j\}_{j=1}^{+\infty}\subset \mathscr F;$ 

$(b)$   a sequence $\{p_j\}_{j=1}^{+\infty}\subset\Omega$ with $p_j\to p_0\in\Omega$ as $j\to+\infty;$

$(c)$ a  sequence $\{\rho_j\}_{j=1}^{+\infty}\subset\mathbb R$ with  $0<\rho_j\to0$ as $j\to+\infty$ 

\noindent such that  the sequence 
$$\phi_j(z):=f_j\circ\exp_{p_j}\big(\rho_j(dL_{p_j})_e(\varepsilon z^T)\big), \ \ \ \    j=1,2,\cdots$$
converges uniformly on compact subsets of $\mathbb C^m$ to a nonconstant holomorphic mapping   $\phi: \mathbb C^m\to N.$  
Here, $z^T$ is the  transpose of $z=(z_1,\cdots,z_m).$
\end{theorem}

In further, we  obtain the following  local version of Theorem \ref{main}.

\begin{cor}\label{mmpp} Let $G$ be an  $m$-dimensional complex Lie group with holomorphic Lie algebra 
$\mathfrak g,$ 
and $\Omega$ a domain in $G.$ Fix  a  basis  
   $\varepsilon=(\varepsilon_1,\cdots, \varepsilon_m)$ of $\mathfrak g.$
Let $N$ be a compact Hermitian manifold. Let $\mathscr F: \Omega\to N$ be a family of holomorphic mappings. 
Then,  $\mathscr F$ is not normal at  $p_0\in\Omega$  if and only if  there exist 

$(a)$ a  sequence $\{f_j\}_{j=1}^{+\infty}\subset \mathscr F;$ 

$(b)$   a sequence $\{p_j\}_{j=1}^{+\infty}\subset\Omega$ with $p_j\to p_0$ as $j\to+\infty;$

$(c)$ a  sequence $\{\rho_j\}_{j=1}^{+\infty}\subset\mathbb R$ with  $0<\rho_j\to0$ as $j\to+\infty$ 

\noindent  such that  the sequence 
$$\phi_j(z):=f_j\circ\exp_{p_j}\big(\rho_j(dL_{p_j})_e(\varepsilon z^T)\big), \ \ \ \    j=1,2,\cdots$$
converges uniformly on compact subsets of $\mathbb C^m$ to a nonconstant holomorphic mapping  $\phi:\mathbb C^m\to N.$  
Here, $z^T$ is the  transpose of $z=(z_1,\cdots,z_m).$
\end{cor}

To see how Theorem B can be recovered by our  result, we treat a special  case where  $G=\mathbb C^m.$
We regard  $\mathbb C^m$ as an additive group with zero element $\textbf{0},$ and  define the left translation 
$L_{p}(z)=p+z$ for  $p, z\in \mathbb C^m.$
Put 
$$\varepsilon=\Big(\frac{\partial}{\partial \zeta_1},\cdots,\frac{\partial}{\partial \zeta_m}\Big)\Big|_{\textbf 0},$$
where $\zeta_1,\cdots,\zeta_m$ stand for   complex coordinates of $\mathbb C^m.$
Since $T^{1,0}_z\mathbb C^m\cong\mathbb C^m$ for all $z\in\mathbb C^m,$  it is immediate that 
$$(dL_{p_j})_{\textbf 0}(\varepsilon z^T)=\Big(z_1\frac{\partial}{\partial \zeta_1}+\cdots+z_m\frac{\partial}{\partial \zeta_m}\Big)\Big|_{p_j}\in T^{1,0}_{p_j}\mathbb C^m.$$
By  the definition of $\exp,$ we obtain 
$\exp_{p_j}(\rho_j(dL_{p_j})_{\textbf0}(\varepsilon z^T))=p_j+\rho_j z,$
which indicates  that Theorem \ref{main} extends Theorem B.

Afterwards, we shall state    a Robinson-Zalcman type heuristic principle of holomorphic families of a  Picard type   property on a complex Lie group.  
Let  $g$ be a holomorphic mapping on a domain  $D$ in a complex manifold.   
Denote by  $P$   a property. 
     Write  $(g, D)\in P,$ if   $g$ is of   $P$ on $D;$ 
          $(g, x)\in \neg P,$   if  $g$ is not of       $P$ 
               in any neighborhood of  $x$ in $D;$ 
                                     and    
 $(g, D)\in \neg P,$  if   $(g, x)\in \neg P$ for all $x\in D.$

Let  $\Omega$ be a domain in a complex Lie group    $G,$ and $N$  a compact Hermitian manifold. 
 Let $f, f_j: \Omega\to N$ be  holomorphic mappings. 
    For $p, p_j\in G,$ put  
   $$\psi(z)=\exp_p(\theta z^T);  \ \ \ \   \  \phi_j(z)=f_j\circ\exp_{p_j}(\theta^j z^T), \ \ \  \  j=1,2,\cdots$$ 
     where $\theta, \theta^j$
are basis of $T^{1,0}_pG, T^{1,0}_{p_j}G,$ respectively.  
With these notations,  we obtain    a Robinson-Zalcman type heuristic principle as follows.

\begin{theorem}\label{ZZZ}  Let $G$ be  an  $m$-dimensional complex Lie group, and $\Delta, \Omega$  domains in $G$ 
with $\Delta\subset\Omega.$ Let $N$ be a compact Hermitian  manifold.   
 Assume that  $P$ is a property of holomorphic mappings   into $N$ such that    

$(a)$ if $(f, \Omega)\in P,$   then $(f, \Delta)\in P$ for any   $\Delta;$

$(b)$ if  $(f, \Delta)\in P,$  then   
$(f\circ\psi, \psi^{-1}(\Delta))\in P$ for any  $p, \theta;$

$(c)$  if    $\{\phi_j; (\phi_j, D_j)\in P\}_{j=1}^{+\infty}$ converges uniformly on  compact subsets of $\mathbb C^m$ to
 a  holomorphic mapping  $\phi: \mathbb C^m\to N$ for  domains $D_1,D_2,\cdots$ in $\mathbb C^m$ such that $D_1\subset D_2\subset\cdots,$  
    then  $(\phi, \mathbb C^m)\in P$ or  $(\phi, \mathbb C^m)\in \neg P;$  
    
$(d)$ if  $(\phi, \mathbb C^m)\in P$ or $(\phi, \mathbb C^m)\in\neg P,$ then $\phi$ is  constant.   

\noindent Then, the family $\mathscr F=\{f: (f,\Omega)\in P\}$ is normal on $\Omega.$  
\end{theorem}

  \section{Exponential Mappings of Complex Lie Groups}\label{sec2}

A complex Lie group $G$ is a complex manifold with a group structure such that its multiplication  $(g, h)\mapsto gh$ and  inverse  $g\mapsto g^{-1}$   are holomorphic for  $g, h\in G.$ 
In this section, we  shall prove  some basic   properties of exponential mappings of complex Lie groups. 
We refer the reader  to  \cite{Bu, Fa, Lee, Ross, V-S-C} for a more in-depth  understanding of  real and  complex Lie groups. 

Let $G$ be a  finite-dimensional complex Lie group with identity element $e.$ We use   $\mathfrak g$ to denote  the 
   \emph{holomorphic  Lie algebra}  of $G$  defined by  $\mathfrak g:=T^{1,0}_eG,$ where  $T^{1,0}_eG$ denotes  the holomorphic tangent space of $G$ at $e.$
 A holomorphic vector field $\tilde X$ on $G$ is called  \emph{left-invariant}, if 
 $(dL_g)_h\tilde X_h=\tilde X_{gh}$ for $g, h\in G,$
 where  $(dL_g)_h$ is    the differential or derivative of  left translation $L_g$ at $h.$ 
 Note that each vector  $X\in\mathfrak g$  generates a unique  left-invariant holomorphic vector field $\tilde X$ on $G$:     
$$\tilde X_g=(dL_g)_e(X)\in T^{1,0}_gG, \ \ \ \   g\in G$$
 with  $\tilde X_e=X.$
 Associate    a  $\mathscr C^1$-continuous  curve $\gamma: \mathbb C\to G$ satisfying    
$$\gamma'=\tilde X_{\gamma}, \ \ \ \     \gamma(0)=e.$$
  It is known  that   such a curve $\gamma$ has to be   existent  uniquely  and     
     holomorphic due   to   ODE theory. 
     For convenience, we shall use   $\gamma_X$ to denote    this solution   associated to $X.$ 

The \emph{exponential mapping}  $\exp: \mathfrak g\to G$ is defined by 
$\exp(X)=\gamma_X(1)$ with  $X\in \mathfrak g,$
where $\gamma_{X}: \mathbb C\to G$ is a   holomorphic curve  satisfying    
$\gamma'_X=\tilde X_{\gamma_{X}}$ with $\gamma_X(0)=e.$
Clearly,   $\exp$ is  holomorphic  on $\mathfrak g.$

\begin{property}\label{p1}
$\{\gamma_X(z)\}_{z\in\mathbb C}$ is a holomorphic  one-parameter   subgroup. 
\end{property}
\begin{proof}
Since  $\gamma_X(0)=e\in\mathfrak g,$ it suffices   to prove    that $\gamma_X(\zeta)\gamma_X(z)=\gamma_X(\zeta+z)$ for $\zeta, z\in\mathbb C.$
Set  
$\alpha_\zeta(z)=\gamma_X(\zeta)\gamma_X(z)$ and    $\beta_{\zeta}(z)=\gamma_X(\zeta+z).$
It is equivalent   to   show   that     $\alpha_\zeta, \beta_\zeta$ solve the following  Cauchy   initial valued problem
\begin{equation}\label{eq}
\gamma'=\tilde X_{\gamma}, \ \ \ \     \gamma(0)=\gamma_X(\zeta)
\end{equation}
  due    to  the uniqueness theorem. 
 By  $\gamma_X(\zeta)\gamma_X(z)=L_{\gamma_X(\zeta)}(\gamma_X(z)),$  we obtain   
 \begin{eqnarray*}
\alpha'_\zeta(z)&=& (dL_{\gamma_X(\zeta)})_{\gamma_X(z)}\gamma'_X(z) \\
&=& (dL_{\gamma_X(\zeta)})_{\gamma_X(z)}\tilde X_{\gamma_X(z)} \\
&=&  (dL_{\gamma_X(\zeta)})_{\gamma_X(z)}(dL_{\gamma_X(z)})_{e}X \\
&=&  (dL_{\gamma_X(\zeta)\gamma_X(z)})_{e}X \\
&=& (dL_{\alpha_\zeta(z)})_{e}X \\
&=& \tilde X_{\alpha_\zeta(z)}. 
 \end{eqnarray*}
Moreover, we have  
 $\alpha_\zeta(0)=\gamma_X(\zeta)\gamma_X(0)=\gamma_X(\zeta)e=\gamma_X(\zeta).$ Hence,   $\alpha_\zeta$ solves    Eq. (\ref{eq}). 
Also, by $\beta_\zeta(0)=\gamma_X(\zeta)$  and $\beta'_\zeta(z)=\gamma'_X(\zeta+z)= \tilde X_{\gamma_X(\zeta+z)}=\tilde X_{\beta_\zeta(z)},$
   we    deduce      that  $\beta_\zeta$   solves   Eq. (\ref{eq}).  Hence, we have  $\alpha_\zeta=\beta_\zeta.$  
\end{proof}

 \begin{property}\label{p2} For any $X\in\mathfrak g,$ we have 
  $$\gamma_X(z)=\exp(zX), \ \ \ \ \      z\in\mathbb C.$$
  \end{property}
\begin{proof}
   Set  $\alpha(t)=\gamma_X(zt),  \beta(t)=\gamma_{zX}(t)$ for $t\in\mathbb C.$
 Since $\exp(zX)=\gamma_{zX}(1),$   it is   sufficient  to    prove   that  $\alpha(t)=\beta(t)$    by putting     $t=1,$ which 
 is equivalent to show  that  $\alpha, \beta$ solve the same Cauchy  initial value problem. By 
   $$\alpha'(t)=z\gamma'_{X}(zt)=z\tilde X_{\alpha(t)}=(dL_{\alpha(t)})_e(zX),$$ 
  $$\beta'(t)=\widetilde{zX}_{\beta(t)}=(dL_{\beta(t)})_e(zX),$$
  we deduce   that   $\alpha, \beta$ satisfy the same first-order  differential equation.  Since   
$\alpha(0)=\beta(0)=e,$  we conclude that  $\alpha=\beta.$ This completes the proof. 
\end{proof}

\begin{property}\label{p0} For any $X\in\mathfrak g,$ we have 
$$\frac{d}{dz}\exp(zX)=\big(dL_{\exp(zX)}\big)_e(X), \ \ \ \     z\in\mathbb C.$$
\end{property}
     \begin{proof}  By Property \ref{p2}, we have
          $\exp(zX)=\gamma_X(z).$ Then  
        $$\frac{d}{dz}\exp(zX)=\gamma'_X(z)=\tilde X_{\gamma_X(z)}=\big(dL_{\gamma_X(z)}\big)_e(X)=\big(dL_{\exp(zX)}\big)_e(X).$$
     \end{proof}

In general, a vector  $\xi\in T_g^{1,0}G$ with $g\in G$ generates     a unique  left-invariant holomorphic vector field $\tilde\xi$ on $G$:    
  $$\tilde{\xi}_h=(dL_{hg^{-1}})_g(\xi)=(dL_h)_e\circ (dL_{g^{-1}})_g(\xi)\in T^{1,0}_hG, \ \ \ \      h\in G$$
with  $\tilde\xi_g=\xi.$
Similarly, we   define the \emph{exponential mapping} $$\exp_g: \ T_g^{1,0}G\to G, \ \ \ \    g\in G$$
  via $\exp_g(\xi)=\gamma_{\xi}(1)$ for $\xi\in T^{1,0}_gG,$
   where     
$\gamma_{\xi}: \mathbb C\to G$ is a   holomorphic curve    solving   uniquely  the  differential  equation  
 $\gamma'=\tilde\xi_{\gamma}$ with  initial value   $\gamma(0)=g.$
  It is clear    that    $\exp_g$ is   holomorphic  on  $T_g^{1,0}G.$ 
  For convenience,  we  shall  use  $\gamma_\xi$ to denote   this solution   associated to $\xi.$

Applying      the similar  arguments as in the proofs of Properties  \ref{p1} and \ref{p2}, we can also   show   the following   Properties  \ref{pp4} and \ref{p11}.
\begin{property}\label{pp4}
$\{\gamma_{\xi}(z)\}_{z\in\mathbb C}$ is a holomorphic  one-parameter   subgroup. 
\end{property}
 \begin{property}\label{p11} For any $\xi\in T_g^{1,0}G$ with $g\in G,$ we have 
$$\gamma_{\xi}(z)=\exp_g(z\xi), \ \ \ \  z\in\mathbb C.$$
  \end{property}
  
    \begin{property}\label{p9} For any $\xi\in T^{1,0}_gG$ with $g\in G,$  we have 
$$\exp_g(z\xi)=g\exp(zX), \ \ \ \     X=(dL_{g^{-1}})_g(\xi).$$
\end{property}
  \begin{proof} Set
  $\eta(z)=g\exp(zX)=L_g\circ\exp(zX).$
  By virtue of  Property \ref{p11}, we have 
  $\exp_g(z\xi)=\gamma_{\xi}(z).$ 
  Since $\eta(0)=\gamma_{\xi}(0)=g,$ it suffices to prove that $\eta, \gamma_{\xi}$ satisfy  the same first-order differential equation. 
By Property \ref{p0}
\begin{eqnarray*}
\eta'(z)&=&(dL_g)_{\exp(zX)}\circ (dL_{\exp(zX)})_e(X) \\
&=&d\big(L_g\circ L_{\exp(zX)}\big)_e(X) \\
&=&\big(dL_{g\exp(zX)}\big)_e(X) \\
&=& \big(dL_{\eta(z)}\big)_e(X).
 \end{eqnarray*}
On the other hand, we have 
$$\gamma'_{\xi}(z)=\tilde\xi_{\gamma_{\xi}(z)}=(dL_{\gamma_{\xi}(z)})_e\circ (dL_{g^{-1}})_g(\xi)=(dL_{\gamma_{\xi}(z)})_e(X).$$
  Therefore,  $\eta, \gamma_{\xi}$ satisfy  the same first-order differential equation.  The proof is completed. 
  \end{proof}
 
 If  $z=1$ is taken in Property \ref{p9},  then it is immediate that  
  \begin{cor}\label{cry} For any $g\in G,$ we have 
$$\exp_g=L_g\circ\exp\circ(dL_{g^{-1}})_g.$$
\end{cor}

Fix  $\xi\in T^{1,0}_{g}G$ with $g\in G.$   Let 
$$\iota_\xi: \ T^{1,0}_{g}G\to T^{1,0}_\xi (T^{1,0}_gG)$$
be the canonical linear isomorphism,  i.e.,  
$\iota_{\xi}(\eta)=\theta'_\eta(0)$ for $\eta\in T^{1,0}_gG,$
where $\theta_\eta(z)=\xi+z\eta$ is a holomorphic curve in  $T^{1,0}_gG.$  
For   multiple tangent bundles of $G,$   one   regards 
 them  as identical.   
    Equip $G$ with a left-invariant Hermitian metric (such a metric is aways existent, see \cite{Fa, V-S-C}),  i.e., 
 $$\langle\xi, \eta\rangle=\big\langle(dL_h)_g(\xi), (dL_h)_g(\eta)\big\rangle, \ \  \ \   \xi, \eta\in T^{1,0}_gG$$
for any $g, h\in G.$ With this metric,   $(dL_h)_g$ is a linear isometry and so   
 \begin{equation}\label{dL}
 \|(dL_h)_g\|=\|(dL_h)^{-1}_g\|=1, \ \ \ \      g, h\in G.
 \end{equation}
This  metric  on $T^{1,0}G$ induces a natural metric on $T^{1,0}(T^{1,0}G)$ defined by
$$\langle \zeta,  \eta\rangle:=\big\langle \iota_{\xi}^{-1}(\zeta), \iota_{\xi}^{-1}(\eta)\big\rangle, \ \ \ \    \zeta, \eta\in T^{1,0}_\xi (T^{1,0}_gG),$$
which forces    $\iota_\xi$ to be  a linear  isometry and it is therefore  
 \begin{equation}\label{xpp}
\|\iota_\xi\|=\|\iota_\xi^{-1}\|=1, \ \ \ \     \|d\iota_{\xi}\|=\|d\iota^{-1}_{\xi}\|=1.
  \end{equation}
By (\ref{dL}) and (\ref{xpp}),  
 we  may regard  in order  to simplify the formula  that 
  \begin{equation}\label{exp00}
  (d(dL_{h})_g)_\xi=(dL_{h})_g\circ \iota^{-1}_{\xi},
   \end{equation}
  up to a linear  isometry.
 
      Let 
  $$(d\exp_g)_{\xi}: \ T^{1,0}_\xi (T^{1,0}_gG)\to T^{1,0}_{\exp_{g}(\xi)}G$$ 
be  the differential of $\exp_g: T^{1,0}_gG\to G$ at $\xi,$ which  can be  represented     as 
$$(d\exp_g)_\xi(\eta)=\frac{d}{dz}\Big{|}_{z=0}\exp_g\big(\xi+z\iota_{\xi}^{-1}(\eta)\big), \ \ \ \     \eta\in T^{1,0}_\xi (T^{1,0}_gG).
$$
For $d\exp_X,$ the well-known   Duhamel formula (see, e.g., Theorem 5 in Section 1.2,  \cite{Ross}) gives      an explicit  expression
 \begin{equation}\label{exp10}
  d\exp_X=\big(dL_{\exp(X)}\big)_e\circ\frac{Id-e^{-{\rm{ad}}_X}}{{\rm{ad}}_X}\circ\iota^{-1}_{X}, \ \ \ \    X\in\mathfrak g,
     \end{equation}
     where 
 $$\frac{Id-e^{-{\rm{ad}}_X}}{{\rm{ad}}_X}=\sum_{k=0}^{+\infty}\frac{(-1)^k}{(k+1)!}({\rm{ad}}_X)^k$$
 with   ${\rm{ad}}_X(Y)=[X,Y]=XY-YX$ for  $Y\in\mathfrak g.$
Set $X=(dL_{g^{-1}})_g(\xi).$
Using Corollary \ref{cry},  (\ref{exp00}) and (\ref{exp10}),  we derive  that 
  \begin{eqnarray}\label{exp11}
   (d\exp_{g})_\xi&=&(dL_g)_{\exp(X)}\circ \big(dL_{\exp(X)}\big)_e\circ\frac{Id-e^{-{\rm{ad}}_X}}{{\rm{ad}}_X} \\
   &&\circ \iota^{-1}_X\circ (dL_{g^{-1}})_g\circ\iota^{-1}_{\xi}.  \nonumber
   \end{eqnarray}
 
        \begin{property}\label{unm} For any $g\in G,$
$(d\exp_{g})_{\emph{\textbf 0}}$ is a linear  isometry. 
  \end{property}
  \begin{proof}  
Using (\ref{exp11}) (or  combining  Corollary \ref{cry} and (\ref{exp00}) with  (\ref{exp10})),  we obtain  $$(d\exp_g)_{\textbf 0}=
 (dL_{g})_e\circ\iota^{-1}_{\textbf0}\circ (dL_{g^{-1}})_g\circ \iota^{-1}_{\textbf0},$$ which is a linear  isometry. 
  \end{proof}
  
Define 
$$C_{\mathfrak g}:=\sup_{X, Y\in\mathfrak g, \  \|X\|=\|Y\|=1}\|[X, Y]\|,$$ 
which is called  the \emph{structure constant} of $G.$ Note  that   $0\leq C_{\mathfrak g}<+\infty,$ since  $\mathfrak g$ has finite dimension.  By    
$\|{\rm{ad}}_X(Y)\|=\|[X, Y]\|\leq C_{\mathfrak g}\|X\|\|Y\|,$ we get  
\begin{equation}\label{eq12}
\|{\rm{ad}}_X\|\leq C_{\mathfrak g}\|X\|.
\end{equation}

 \begin{theorem}\label{p333} For any $\xi\in T^{1,0}_gG$ with $g\in G,$  we have 
$$\big\|(d\exp_{g})_{\xi}\big\| \leq  \frac{e^{C_{\mathfrak g}\|\xi\|}-1}{C_{\mathfrak g}\|\xi\|}.$$
  \end{theorem}
  
\begin{proof}
It yields from   (\ref{dL}), (\ref{xpp}), (\ref{exp11}) and (\ref{eq12}) that 
\begin{eqnarray*}
\big\|(d\exp_{g})_{\xi}\big\| &\leq&\big\|(dL_g)_{\exp(X)}\big\|\cdot  \big\|\big(dL_{\exp(X)}\big)_e\big\|
\cdot \bigg\|\frac{Id-e^{-{\rm{ad}}_X}}{{\rm{ad}}_X}\bigg\| \\
&& \cdot \big\|\iota^{-1}_X\big\|\cdot \big\|(dL_{g^{-1}})_g\big\|\cdot \big\|\iota^{-1}_{\xi}\big\| 
 \\
&\leq& \bigg\|\frac{Id-e^{-{\rm{ad}}_X}}{{\rm{ad}}_X}\bigg\| \\
&\leq& \sum_{k=0}^{+\infty}\frac{1}{(k+1)!}\big(C_{\mathfrak g}\|X\|\big)^k \\
&=& \frac{e^{C_{\mathfrak g}\|\xi\|}-1}{C_{\mathfrak g}\|\xi\|},
\end{eqnarray*}
where $X=(dL_{g^{-1}})_g(\xi).$
\end{proof}

\section{Proofs  of  Theorems \ref{main} and \ref{ZZZ}}

 Let $\mathscr C(M, N)$ denote  the set of all continuous mappings between complex manifolds $M,N.$ 
Equip $\mathscr C(M,N)$ with  the compact-open topology. 
A family $\mathscr F\subset \mathscr C(M,N)$ is called  \emph{normal}, if and only if each sequence of $\mathscr F$ contains a subsequence which 
 is either relatively compact in $\mathscr C(M,N)$ 
 or compactly divergent (Definition 1.1,   \cite{Wu}).  
 Let  $\mathscr H(M,N)\subset \mathscr C(M, N)$  
 denote     the set of all holomorphic  mappings between   $M,N.$ Anyway, 
once a distance function is chosen on $N,$ 
a sequence $\{f_j\}_{j=1}^{+\infty}\subset \mathscr C(M, N)$ converges to  an $f\in \mathscr C(M, N)$
 if and only if $\{f_j\}_{j=1}^{+\infty}$ converges uniformly on compact subsets  to 
$f.$  Hence, a basic fact 
 derived from  Weierstrass theorem shows that 
 $\mathscr H(M,N)$ is closed in $\mathscr C(M, N).$   
 When  $N$ is  a compact Hermitian manifold, 
 using  Ascoli-Arzel\`a theorem  or 
    a result of Wu (Lemma 1.1, \cite{Wu}),  
    $\mathscr F$ is normal  on $M$ as long as  $\mathscr F$ is equicontinuous on $M.$

 To prove Theorem \ref{main}, we    extend    Marty's criterion  \cite{FM} for normality.  
 Let $f: M\to N$ be a holomorphic  mapping between   Hermitian manifolds $M, N.$
The \emph{operator norm} $\|df\|$ of  differential $df$ is defined by 
$$\|df_x\|=\sup_{\xi\in T^{1,0}_xM}\frac{\|df_x(\xi)\|_N}{\|\xi\|_M}, \ \ \ \  x\in M.$$
\ \ \ \    We provide  a slight extension of   Marty's criterion  as follows.

\begin{lemma}\label{normal} Let $M,N$ be Hermitian manifolds with that $N$ is compact. Let $\mathscr F:M\to N$ be a family of holomorphic mappings. Then, $\mathscr F$ is normal on $M$ if and only if  $\{\|df\|\}_{f\in\mathscr F}$ is locally uniformly bounded on $M.$
\end{lemma}

\begin{proof} 
We first prove  the sufficiency. Let    $x_0\in M$ be an arbitrary point, and 
 take a small neighborhood $U(x_0)$ of $x_0$ in $M.$
  For  a  fixed   $x\in U(x_0),$  we treat  
     a $\mathscr C^1$-continuous curve $\gamma:[0,1]\to M$ which connects  $x_0$ with $x.$
     It is evident  that 
     $f\circ\gamma:[0,1]\to N$ is a $\mathscr C^1$-continuous curve connecting $f(x_0)$ with $f(x).$
Note that 
$d_M(x_0,x)=\inf_{\gamma}\|\gamma\|_M,$
where $\|\gamma\|_M$ is  the length of $\gamma.$ Then   
\begin{eqnarray*}
d_N(f(x_0), f(x))&\leq& \|f\circ\gamma\|_N \\
&=& \int_0^1\|(f\circ\gamma)'\|_N(t)dt \\
&\leq&  \int_0^1\|df_{\gamma(t)}\|\cdot\|\gamma'\|_M(t)dt, 
\end{eqnarray*}
where  $\|f\circ\gamma\|_N$ is the length of $f\circ\gamma.$ Utilizing  the conditions,  $\{\|df\|\}_{f\in\mathscr F}$ is uniformly bounded on $U(x_0).$ So, there exists  a constant $C>0,$ independent of $x\in U(x_0)$ and  $f\in\mathscr F,$ such that
$$\sup_{f\in\mathscr F}\sup_{x\in U(x_0)}\|df_x\|\leq C.$$
Therefore, we obtain 
$$d_N(f(x_0), f(x))\leq  C\int_0^1\|\gamma'\|_M(t)dt=C\|\gamma\|_M.$$
By the arbitrariness of $\gamma,$ it is immediate that 
 $$d_N(f(x_0), f(x))\leq  C d_M(x_0,x), \ \ \ \     f\in\mathscr F,$$
 which deduces that $\mathscr F$ is equicontinuous at $x_0,$ and hence  equicontinuous on $M.$
 Since $N$ is compact, $\mathscr F$  is normal on $M.$
 
 We then prove the necessity. On the contrary, we assume that $\{\|df\|\}_{f\in\mathscr F}$ is not locally uniformly bounded on $M.$
 Then, there are   geodesic balls $B, B'$ with $\overline B\subset B'\subset M,$ 
  and     sequences $\{x_j\}_{j=1}^{+\infty}\subset \overline{B},$  
        $\{f_j\}_{j=1}^{+\infty}\subset\mathscr F$ such that   
 \begin{equation}\label{ef}
 \lim_{j\to+\infty}\|(df_j)_{x_j}\|=+\infty.
 \end{equation}
 Using the compactness of  $\overline{B},$   there exists  a  subsequence $\{x_{j_k}\}_{k=1}^{+\infty}\subset \{x_j\}_{j=1}^{+\infty}$ 
  which  converges   to  a limit point $x_0\in \overline{B}.$ 
        Take  a     neighborhood  $U$ of $x_0$ such that  $\overline{U}\subset B'.$ 
               When $j$ is sufficiently large,  we have  $x_j\in U.$  
  Moreover,   since  $\mathscr F$ is normal on $M,$  
                            there is    a subsequence $\{f_{j_k}\}_{k=1}^{+\infty}\subset \{f_j\}_{j=1}^{+\infty}$ which 
                                                            converges    uniformly  to a   limit    mapping $f_0: \overline{U}\to N.$ 
                                                              Invoking   
     Weierstrass theorem, $f_0$ is holomorphic on $U$ and we receive  
                $$\lim_{k\to+\infty}\|(df_{j_k})_x\|=\|(df_0)_x\|\not=+\infty, \ \ \ \   x\in U$$
                 due to the continuity  of $\|(df_0)_x\|, \|(df_{j_k})_x\|$ in $x.$
                                  Since $x_{j_k}\to x_0$ as $k\to+\infty,$ we conclude that   
                $$\lim_{k\to+\infty}\|(df_{j_k})_{x_{j_k}}\|=\lim_{k\to+\infty}\|(df_{0})_{x_{j_k}}\|=\|(df_0)_{x_0}\|\not=+\infty.$$
     However, it contradicts with (\ref{ef}). This completes the proof. 
\end{proof}

\emph{Proof of Theorem $\ref{main}$}

\begin{proof} Equip $G$ with a left-variant Hermitian metric such that 
 $\{\varepsilon_1,\cdots, \varepsilon_m\}$ forms  an orthonormal  basis of $\mathfrak g.$
Assume that  $\mathscr F$ is not normal on $\Omega.$  Using    Lemma \ref{normal},    for any sequence $\{s_j\}_{j=1}^{+\infty}\subset \mathbb R$
 with $s_j\geq1$ for all $j,$ there exist    a compact subset    $K\subset\Omega,$
 a sequence $\{\tilde p_j\}_{j=1}^{+\infty}\subset K,$  a sequence $\{\tilde\xi_j\}_{j=1}^{+\infty}$ with $\tilde\xi_j\in T^{1,0}_{\tilde p_j}G$ and $\|\tilde\xi_j\|_G=1,$    
as well as   a sequence $\{f_j\}_{j=1}^{+\infty}\subset \mathscr F,$ such  that   
    \begin{equation}\label{geq}
  \big\|(df_j)_{\tilde p_j}(\tilde\xi_j)\big\|_N\geq j^{s_j}, \ \ \ \  j=1,2,\cdots
    \end{equation}
 By the compactness of  $K,$     there exists a convergent  subsequence of  $\{\tilde p_j\}_{j=1}^{+\infty}.$   
  Without loss of generality, it is convenient to  suppose   that  
 $\tilde p_j\to p_0\in K$ as $j\to+\infty$ and $d_G(p_0, \tilde p_j)< 1/(2j)$ for all $j,$  in which $d_G(\cdot,\cdot)$ is   the Riemannian distance  on $G.$  Let   $B_{G}(p_0, r)$ be the geodesic ball centered at $p_0$ with radius $r$ in $G.$ 
   For  $j=1,2,\cdots,$ put 
   $$M_j=\max_{g\in \overline{B_{G}(p_0, 1/j)}}\max_{\xi\in T^{1,0}_{g}G, \|\xi\|_G=1}\big\|(df_j)_g(\xi)\big\|_N.$$
Taking   $p_j\in \overline{B_{G}(p_0, 1/j)}$ and $\xi_j\in T^{1,0}_{p_j}G$ with $\|\xi_j\|_G=1$  such that   the above    expression is maximized, i.e., 
\begin{equation}\label{eqM}
    M_j=\big\|(df_j)_{p_j}(\xi_j)\big\|_N, \ \ \ \     j=1,2,\cdots
\end{equation}
It is apparent      that   $p_j\to p_0$ as $j\to+\infty,$ due to   $d_G(p_0, p_j)\leq1/j$  inherent  in  the definition of  $M_j$.  Thus, we assume without loss of generality that $p_j\in\Omega$ for all $j.$
 Define  
$$\rho_j=\frac{1}{M_j}, \ \ \ \   j=1,2,\cdots$$
It yields from  (\ref{geq}) and (\ref{eqM}) that  
$$M_j \geq \big\|(df_j)_{\tilde p_j}(\tilde\xi_j)\big\|_N
\geq j^{s_j},  \ \ \ \   j=1,2,\cdots$$
Since $s_j\geq1$ for all $j,$  we arrive at  
\begin{equation}\label{alal}   
  0<\rho_j\leq j^{-s_j} \to0 \ \ \ \  \ \  (j\to+\infty).
  \end{equation}
 \ \ \ \     Let us treat    the differential of $\exp_{p_j}$ at $\textbf 0$:
    $$(d\exp_{p_j})_{\textbf0}: \  T^{1,0}_{\textbf 0}(T^{1,0}_{p_j}G)\to T^{1,0}_{p_j}G, \ \ \ \    j=1,2,\cdots$$
     Define 
 $$\phi_j(z)=f_j\circ\exp_{p_j}\big(\alpha_j(z)\big), \ \ \ \    j=1,2,\cdots$$ 
 where    
   \begin{eqnarray*}
  \alpha_j(z) &=& \rho_j(dL_{p_j})_e(\varepsilon z^T).
    \end{eqnarray*}
 Since 
 $f_j, \exp_{p_j}, (dL_{p_j})_e$ are  holomorphic, 
it is obvious   that  $\phi_j$ is    holomorphic. 
 Moreover,  since  $(dL_{p_j})_e$ is a  linear isometry by   (\ref{dL}),   we infer     that  $\alpha_j:\mathbb C^m\to \mathfrak g$ is a  linear isomorphism  with  
$$\|\alpha_j\|=\rho_j, \ \ \ \    j=1,2,\cdots$$
  Let      
 $B_{\mathfrak g}(r)$ be    the geodesic ball centered at $\textbf 0$ with radius $r$ in $\mathfrak g.$ 
 Since  $\exp_{p_j}$ is  locally  biholomorphic  at $\textbf 0$ by Property \ref{unm}, and $(dL_{p_j})_e$ is  linearly  isometric, 
  there exists a number $r_j>0$ such that 
  $$\exp_{p_j}\big((dL_{p_j})_e\big(\overline{B_{\mathfrak g}(r_j)}\big)\big)\subset  B_{G}\big(p_0, 1/j\big).$$ 
  It is obvious  that     $r_j\to 0$ as $j\to+\infty.$ 
     Without loss of generality, we  assume that  
 $0<r_j<1$ for all $j.$ 
   Applying  the linear isometry of   $(dL_{p_j})_e,$  we have  
   \begin{equation}\label{wss}   
   \exp_{p_j}\big(\alpha(z)\big)\in  B_{G}\big(p_0, 1/j\big), \ \ \ \  \|z\|<\frac{r_j}{\rho_j}.
   \end{equation}
 Take 
 $$s_1=1;  \ \ \   \  \   s_j=1-\frac{\log r_j}{\log j}>1,  \ \ \ \  j=2, 3, \cdots$$
By the identity 
$x^{1/\log x}=e$ for  $x>0,$
 we conclude   from  (\ref{alal})  that $r_j\rho^{-1}_j\geq j$ for all $j\geq 2.$ By virtue of  (\ref{wss}),  we deduce that 
     $\phi_j$  is   well-defined on the ball $\{z\in\mathbb C^m: \|z\|<j\}$ with $j\geq2.$ Note  that this  ball  is    larger and larger    as $j$ increases.  
   Set 
   $$x_j(z)=\exp_{p_j}(\alpha_j(z)), \ \ \ \    j=1,2,\cdots$$
 As we have  noted that $d_G(p_0,  p_j)< 1/j$ for all $j.$ 
Fix a  number  $r_0>0.$ When  $\|z\|\leq r_0$ and $j$ is  large  such that $j>r_0,$ we infer   from Theorem \ref{p333}, (\ref{dL})-(\ref{exp00}) and (\ref{eqM})  that    
  \begin{eqnarray*}
 \|(d\phi_{j})_z\|  
 &=&\big\|d\big(f_j\circ\exp_{p_j}\circ\alpha_j\big)_z\big\|  \\
&\leq&  \big\|(df_j)_{x_j(z)}\big\|\cdot\big\|(d\exp_{p_j})_{\alpha_j(z)}\big\|\cdot\|(d\alpha_j)_z\|      \\
&\leq & \rho_jM_j \big\|(d\exp_{p_j})_{\alpha_j(z)}\big\|\cdot \|z\| \\
&\leq& r_0\frac{e^{C_{\mathfrak g}\|\alpha_j(z)\|_G}-1}{C_{\mathfrak g}\|\alpha_j(z)\|_G} \\
&=&r_0\frac{e^{C_{\mathfrak g}\rho_j\|z\|}-1}{C_{\mathfrak g}\rho_j\|z\|} \\
&\leq& r_0\frac{e^{C_{\mathfrak g}\rho_j r_0}-1}{C_{\mathfrak g}\rho_j r_0}  \\
&\to& r_0 \ \ \ \  \ \  (j\to+\infty).
 \end{eqnarray*}
 To conclude, $\phi_j$ is well-defined on  larger and larger balls in $\mathbb C^m$ as $j$ increases,  with    bounded differential on compact subsets. 
 Using     Lemma \ref{normal},  we derive  that $\{\phi_j\}_{j=1}^{+\infty}$  is normal, i.e.,  there  is  a subsequence $\{\phi_{j_k}\}_{k=1}^{+\infty}\subset\{\phi_j\}_{j=1}^{+\infty}$ such that 
 which      converges uniformly on  compact subsets of $\mathbb C^m$ 
      to a  holomorphic mapping  $\phi: \mathbb C^m\to N.$
Finally, we remain to examine   that $\phi$ is  not constant.
Since  $(d\exp_{p_j})_{\textbf 0}$ is a linear isometry   due  to  Property \ref{unm}, we see that 
$$(d\exp_{p_j})_{\textbf 0}\circ (d\alpha_{j})_{\textbf 0}: \ T^{1,0}_{\textbf 0}\mathbb C^m\to T^{1,0}_{p_j}G, \ \ \ \    j=1,2,\cdots$$ are   linear isomorphisms. Therefore,   there exists   a unique vector $t_j\in T^{1,0}_{\textbf 0}\mathbb C^m$ such that 
$(d\exp_{p_j})_{\textbf 0}\circ (d\alpha_{j})_{\textbf 0}(t_j)=\rho_j\xi_j$ for all $j.$
 By this with (\ref{eqM}), we obtain for  $j=1,2,\cdots$
 \begin{eqnarray*}
\big\|(d\phi_{j})_{\textbf 0}(t_j)\big\|_N  &=& \big\|d\big(f_j\circ\exp_{p_j}(\alpha_j)\big)_{\textbf 0}(t_j)\big\|_N  \\
&= & \big\|(df_{j})_{x_j(\textbf{0})}\circ (d\exp_{p_j})_{\alpha_j(\textbf 0)}\circ (d\alpha_j)_{\textbf 0}(t_j)\big\|_N \\
&= & \big\|(df_{j})_{p_j}\circ (d\exp_{p_j})_{\textbf 0}\circ (d\alpha_j)_{\textbf 0}(t_j)\big\|_N \\
&= & \rho_j\big\|(df_{j})_{p_j}(\xi_j)\big\|_N \\
&=&  \rho_jM_j \\
&=&  1,
 \end{eqnarray*}
which  implies  that  $\|(d\phi)_{\textbf 0}\|\not=0.$
Hence,   $\phi$ is  not a constant mapping. 

Assume that  the conditions $(a), (b)$ and $(c)$ are satisfied, we show that $\mathscr F$ is not normal at $p_0.$
Conversely,  assume that  $\mathscr F$ is  normal at $p_0.$  
By  Lemma \ref{normal},  
there exist  $\delta, M>0$ such that $\overline{B_G(p_0, \delta)}\subset\Omega$ and 
$$\sup_{f\in\mathscr F}\|df_g\|\leq M, \ \ \ \          g\in\overline{B_G(p_0, \delta)}.$$
Fix  $z\in\mathbb C^m.$ Since $p_j\to p_0,$  $0<\rho_j\to 0$ as $j\to+\infty,$ when $j$ is   large enough,  we have 
$\exp_{p_j}(\alpha_j(z))\in \overline{B_G(p_0, \delta)}.$
So, it yields  from Theorem \ref{p333} and (\ref{dL})-(\ref{exp00}) that for $j$ large  
 \begin{eqnarray*}
\|(d\phi_j)_z\|&\leq&\|df_j\|_{x_j(z)} \cdot\big\|(d\exp_{p_j})_{\alpha_j(z)}\big\|\cdot\|(d\alpha_j)_z\| \\
&\leq& M\rho_j\|z\| \frac{e^{C_{\mathfrak g}\|\alpha_j(z)\|_G}-1}{C_{\mathfrak g}\|\alpha_j(z)\|_G} \\
&\leq& M\rho_j\|z\|\frac{e^{C_{\mathfrak g}\rho_j\|z\|}-1}{C_{\mathfrak g}\rho_j\|z\|} \\\
&\to& 0 \ \ \ \  \  \  (j\to+\infty).
 \end{eqnarray*}
That is,   
$$\|d\phi_z\|=\lim_{j\to+\infty}\|(d\phi_j)_z\|=0.$$
But, the arbitrariness of $z$ implies  that $\phi$ is constant, a contradiction. 
\end{proof}

By the arguments in the proof of Theorem \ref{main}, it is immediate to deduce that Corollary \ref{mmpp} is true.

\emph{Proof of Theorem $\ref{ZZZ}$}

\begin{proof}  Equip $G$ with a left-invariant Hermitian metric.  
 Conversely,  assume that  $\mathscr F$ is not normal at some point  $p_0\in\Omega.$
  By     Corollary   \ref{mmpp},   there exist     sequences $\{f_j\}_{j=1}^{+\infty}\subset \mathscr F,$  
     $\{p_j\}_{j=1}^{+\infty}\subset \Omega$ with $p_j\to p_0$ as $j\to+\infty,$    
      $\{\rho_j\}_{j=1}^{+\infty}\subset\mathbb R$ with  $0<\rho_j\to0$ as $j\to+\infty,$  
          such that    $\phi_j(z)=f_j\circ\exp_{p_j}(\theta^j z^T)$ 
   converges uniformly on  compact subsets of $\mathbb C^m$ to a  nonconstant holomorphic mapping  $\phi: \mathbb C^m\to N,$ 
  where   $\theta^j=\rho_j(dL_{p_j})_e(\varepsilon)$ gives a basis of $T^{1,0}_{p_j}G$ for  each $j.$ 
Take a  number $\delta>0$ such that $\overline{\Delta}\subset\Omega,$ 
where $\Delta=B_G(p_0, \delta)$ is the geodesic ball centered at $p_0$ with radius $\delta$ in $G.$ 
By $(a),$ we have $(f_j, \Delta)\in P.$ Further, we get   from $(b)$ that  $(\phi_j, D_j)\in P,$ 
where $D_j=x^{-1}_j(\Delta)$ 
and  $x_j(z)=\exp_{p_j}(\theta^j z^T).$ 
Since  $0<\rho_j\to0$ as $j\to+\infty,$     $D_j$ is larger and larger   
as $j$ increases, which leads to that   
$\{\phi_j; (\phi_j, D_j)\in P\}_{j=1}^{+\infty}$ 
converges uniformly on  compact subsets of $\mathbb C^m$ to  $\phi: \mathbb C^m\to N.$ 
By    $(c),$ 
we deduce  that   $(\phi, \mathbb C^m)\in P$
 or $(\phi, \mathbb C^m)\in\neg P.$ 
  However, 
 $\phi$ reduces to a constant mapping due to $(d),$
  a   contradiction.   
 \end{proof}

\noindent\textbf{Conflict of Interest Statement.} The authors declare that there is no conflict of interest.

\end{document}